\documentclass[12pt,leqno,fleqn]{amsart}
\usepackage[color=green!15,bordercolor=red,linecolor=red!30]{todonotes}  
\usepackage{ulem}
\usepackage{amsmath,amstext,amsthm,amssymb,amsxtra} 
\usepackage{txfonts}
\usepackage[T1]{fontenc}
\usepackage{lmodern}

\usepackage{mathtools} 

\usepackage{float}
\usepackage[backrefs]{amsrefs}
\usepackage{hyperref} 
\hypersetup{
    colorlinks=true,       
    linkcolor=blue,          
    citecolor=magenta,        
    filecolor=magenta,      
    urlcolor=cyan,          
}

\newtheorem{theorem}[equation]{Theorem}
\newtheorem{lemma}[equation]{Lemma}
\newtheorem{corollary}[equation]{Corollary}

\newtheorem{definition}{Definition} 

\newcommand{\norme}[1]{ \left\| #1 \right\|}
\newcommand{\rt}{\rightarrow}

\newcommand{\tvec}{\mathcal{T}_{\mathcal Q,r}} 
\newcommand{\ral}{\mathbb{R}}
\newcommand{\nint}{\displaystyle\int}
\newcommand{\midd}[1]{\left| #1 \right|}
\newcommand{\nat}{\mathbb{N}}
\newcommand{\zat}{\mathbb{Z}}

\title{A Characterization of Two Weight Inequalities for a Vector-Valued Operator}
\author[J. Scurry]{James Scurry}
\thanks{Research supported in part by NSF DMS Grant \# 1001098}
\address{ School of Mathematics, Georgia Institute of Technology, Atlanta GA 30332, USA}
\email {jscurry3@math.gatech.edu}

\begin{document}
\begin{abstract}
We give a characterization of the two-weight inequality for a simple
vector-valued operator. Special cases of our result have been
considered before in the form of the weighted Carleson embedding
theorem, the dyadic positive operators of Nazarov, Treil, and Volberg
\cite{NTV} in the square integrable case, and Lacey, Sawyer,
Uriarte-Tuero \cite{lsut} in the $L^p$ case. The main technique of
this paper is a Sawyer-style argument and the characterization is for
$1 < p < \infty$. 
\end{abstract} 
\maketitle
%
%
%
\section{Introduction}
\indent Our focus is on two weight inequalities. We study a simple vector-valued operator $\mathcal{T}_{\mathcal Q,r}$ defined by a sparse collection of cubes $\mathcal{Q}$ and an exponent $1 \le r < \infty$, namely we take
\begin{equation*}
\mathcal{T}_{\mathcal Q,r}f(x) = \left( \sum_{I \in \mathcal Q} |\langle f \rangle_I|^r \mathbf{1}_{I}(x) \right)^{\frac{1}{r}}
\end{equation*}
\noindent for $f \in L^1_{\rm{loc}}(\mathbb{R}^n)$. The aim of our efforts is to give a necessary and sufficient condition for the two weight inequality of $\mathcal{T}_{\mathcal Q,r}$ to hold when $1 < r <\infty$. The main result of this chapter may be formulated as follows:

\begin{theorem} \label{thm: t.mainch1}
Suppose $w$ and $\sigma$ are weights and $1 < r,p < \infty$ with $\mathcal Q$ a sparse collection of cubes. Then we have $\lVert \mathcal{T}_{\mathcal Q,r}(\cdot \sigma) \rVert_{L^p(\sigma) \rightarrow L^p(w)}$ if and only if there are $\mathcal L$ and $\mathcal L_{\ast}$ such that:
\begin{align}
\sup_{Q} \int_{Q} \mathcal{T}_{\mathcal Q,r}(\mathbf{1}_Q \sigma)(x)^p w &\le \mathcal L \sigma(Q) \label{testingo} \\
\sup_{\mathbf{a} } \sup_{Q} \int_{Q} \mathbf{U}_{\mathcal Q}(\mathbf{1}_{Q} \mathbf{a} w)(x)^{p'} \sigma &\le \mathcal L_{\ast} w(Q) \label{testingd}
\end{align}
where $\mathbf{U}_{\mathcal Q}$ is an appropriate `dual' operator (which we define later) and where the first supremum for $\mathbf{U}_{\mathcal Q}$ is taken over all sequences of functions $\mathbf{a}$ such that $\lVert \mathbf{a} \rVert_{\ell^r} = 1$.
\end{theorem}
\indent Special cases of our theorem have been considered before. Notably, when $p=r$ and $w = \sigma$ we obtain the weighted Carleson embedding theorem:
\begin{theorem}
[Weighted Carleson Embedding Theorem] \label{thm: t.carleson1}
Let $w$ be a weight on $\mathbb{R}^n$ and $\{ \tau_J \}_{J \in \mathcal D}$ a collection of nonnegative numbers. Then we have
\begin{align*}
\sup_{I} \frac{1}{w(I)} \sum_{J \subset I} \tau_J &\lesssim 1
\end{align*}
if and only if 
\begin{align}
\sup_{\substack{f \in L^p(w) \\ \lVert f \rVert_{L^p(w)} = 1}} \sum_{J \in \mathcal D} (\langle f \rangle^w_J)^p \tau_J &\lesssim 1.\label{e.carleson}  
\end{align}
\end{theorem}
\noindent Theorem \ref{thm: t.carleson1} is a fundamental result in two weight theory. For positive operators, the relationship between Theorem \ref{thm: t.carleson1} and the corresponding two weight inequality is very strong. The two weight inequality for the maximal function is equivalent to Theorem \ref{thm: t.carleson1} and the characterization of weighted inequalities for discrete positive operators can be reduced to Theorem \ref{thm: t.carleson1}, see \cite{treil}. The connection is less clear for operators without a positive kernel, but if $p=2$ then Theorem \ref{thm: t.carleson1} can be used to give the two weight inequality for the dyadic square function and Haar multipliers (see \cite{NTV}). Our Theorem \ref{thm: t.mainch1} generalizes Theorem \ref{thm: t.carleson1}, reducing to a special case of (\ref{e.carleson}) when $p=r$. \\
\indent Further, for $r=1$ and $p=2$, \cite{NTV} gave a characterization of the operator $\mathcal {T}_{\mathcal Q,r}$. This result was later extended to $p \neq 2$ by \cite{lsut}. A crucial difference between the two papers was that \cite{NTV} used a Bellman function technique while \cite{lsut} constructed a more flexible argument. We rely on the methods presented in \cite{lsut}, noting Theorem \ref{thm: t.mainch1} follows largely from their argument but not directly from their results. \\
\indent We mention the operators $\mathcal{T}_{\mathcal Q,r}$ have also received attention with respect to one weight inequalities. 
The arguments of \cite{cump} imply the following:
\begin{theorem} Let $\mathcal Q$ be a sparse collection of cubes with $1 < r, p < \infty$ and $w \in A_p$. Then we have 
\begin{align}
\lVert \mathcal T_{\mathcal Q,r} \rVert_{L^p(w) \rightarrow L^p(w)} &\lesssim [w]_{A_p}^{\max \left\{ \frac{1}{r}, \frac{1}{p-1} \right\}} \label{e.cump}.
\end{align}
\end{theorem}
\noindent
Using a decomposition theorem of A. Lerner in conjunction with (\ref{e.cump}) the authors of \cite{cump} were able to deduce sharp strong-type inequalities for the vector-valued maximal function and dyadic square function. Later, A. Lerner used a similar argument to extend the square function result to the intrinsic square function. Applying these type of arguments together with Theorem \ref{thm: t.mainch1} and Sawyer's theorem for the maximal function we obtain the following
\begin{corollary} \label{c.twtsq}
Suppose $w$ and $\sigma$ are two weights with $1 < p,r < \infty$. Assume 
the testing conditions (\ref{testingo}) and (\ref{testingd}) are satisfied with constants independent of the sparse collection $\mathcal Q$. Additionally, suppose $M(\cdot \sigma)$ satisfies
\begin{align*}
\int_{Q} M(\mathbf{1}_{Q} \sigma)(x)^p w &\lesssim \sigma(Q).
\end{align*}
Then $\mathbf{M}_{r}(\cdot \sigma)$ is bounded from $L^p(\sigma)$ to $L^p(w)$ and if $r=2$, $S(\cdot \sigma)$ is bounded from $L^p(\sigma)$ to $L^p(w)$.  
\end{corollary}
\indent The remainder of this chapter is structured as follows. In Section 2 we introduce certain definitions and theorems which will be useful for us. The subsequent section deals with several preliminary results and Section 4 contains the bulk of our argument for Theorem \ref{thm: t.mainch1}. 

\section{Initial Concepts}
Throughout the remainder of this chapter we assume $1 < r < \infty$. Recall, for $\mathcal Q$ a sparse collection of cubes and $\mathbf{g} = \left\{ g_I \right\}_{I \in \mathcal Q}$ a collection of measurable functions we set
\begin{align*}
\mathbf{U}_{\mathcal Q}(\mathbf{g})(x) &= \sum_{I \in \mathcal Q} \langle g_I \rangle_I \mathbf{1}_{I}(x).
\end{align*}
\noindent We also consider an operator $\mathbf{T}_{\mathcal Q,r}$ which allows us to overcome the non-linearity of $\mathcal{T}_{\mathcal Q, r}$:
\begin{definition}
Let $f \in L^1_{\rm{loc}}(\mathbb{R}^n)$ and $1 < r < \infty$. We set
\begin{align*}
\mathbf{T}_{\mathcal Q, r}(f)(x) &= \left\{ \langle f \rangle_I \mathbf{1}_{I}(x) \right\}_{I \in \mathcal Q}.
\end{align*}
\end{definition}
\noindent Then we have 
\begin{align*}
\nint_{\ral^n} \mathcal{T}_{ \mathcal{Q}, r}(f \sigma)(x)^p w  &= 
\nint_{\ral^n} \norme{\mathbf{T}_{\mathcal Q, r}(f\sigma)}_{\ell^{r}}^p w  \\ &=
\nint_{\ral^n} \langle \mathbf{T}_{\mathcal{Q}, r }(f \sigma), \mathbf{a} w  \rangle_{\ell^r} dx \\
&= \nint_{\ral^n} \langle f \sigma, \mathbf{U}_{\mathcal Q}(\mathbf{a} w) \rangle_{\ell^r} dx. 
\end{align*}  
\noindent Consequently, $\mathbf{U}_{\mathcal Q}$ can be loosely considered as the dual operator to $\mathcal{T}_{\mathcal Q,r}$. Further, we define certain restrictions of $\mathcal{T}_{\mathcal Q,r}$:
\begin{definition}
Suppose $\mathcal Q$ is a sparse collection of cubes and $1 < r < \infty$. For $Q \subset \mathbb{R}^n$, we have
\begin{align*}
\mathcal{T}_{\mathcal Q, r ,Q}^{\rm{in}}f(x) &=  \left( \sum_{ \substack{I \subseteq Q \\ I \in \mathcal{Q}}} | \langle f \rangle_I |^r \mathbf{1}_{I}(x) \right)^{\frac{1}{r}} , \\
\mathcal{T}_{\mathcal Q, r, Q}^{\rm{out}}(f)(x) &= \left(\sum_{\substack{Q \subset I \\ I \in \mathcal{Q}}} | \langle f \rangle_I|^r \mathbf{1}_{I}(x) \right)^{\frac{1}{r}}.
\end{align*} 
\end{definition}
\indent Now we consider a Whitney covering lemma whose statement we borrow from \cite{lsut} and the universal maximal estimate: 
\begin{lemma} \label{whitney} For each $k$ there exists a collection $\mathcal{Q}_k$ of disjoint cubes satisfying:
\begin{align} 
\Omega _k = \bigcup _{Q\in \mathcal Q_k} Q,  & & \textup{}
\\ \label{e.Whit}
Q^{(1)} \subset \Omega _k\,, \ Q^{(2)}\cap \Omega _k ^{c} \neq \emptyset, & & \textup{}
\\ \label{e.fo}
\sum _{Q\in \mathcal Q_k} \mathbf 1_{Q^{(1)}} \lesssim \mathbf 1_{\Omega _k} ,
& &\textup{}
\\ \label{e.crowd}
\sup _{Q\in \mathcal Q_k} {} {\sharp} \left\{Q'\in \mathcal Q_k \;:\; Q'\cap Q ^{ (1)}\neq \emptyset \right\} \lesssim 1 \,, & &  \textup{}
\\ \label{e.nested}
Q\in \mathcal Q_k\,,\ Q'\in \mathcal Q_l \,,\ Q\subsetneqq Q'\quad \textup{ } \quad 
k> l \,. & &\textup{}.
\end{align} 
\end{lemma}

\begin{theorem} \label{maximal}
Let $\mu$ be a weight and $1 < s \leq \infty$. For $g \in L^s(\omega)$, define
\[ M^{\mu}g(x) = \displaystyle\sup_{\substack{ Q \in \mathcal D \\ Q \ni x}} \langle \midd{g} \rangle^{\mu}_Q .\] 
Then $M^{\mu}:L^s(\mu) \rt L^s(\mu)$ is a bounded operator. 
\end{theorem}
\noindent The proofs of Lemma \ref{whitney} and Theorem \ref{maximal} are standard and we omit them, but relevant arguments can be found in \cite{lsut} and \cite{stein}. For the proof of Corollary \ref{c.twtsq}, we will also need two additional theorems. The first is a decomposition theorem from \cite{lerner2} by A. Lerner:
\begin{theorem} \label{thm: t.lernerch1}
Let $f \in L^1_{\rm{loc}}(\ral^n)$ and let $Q$ be a fixed cube. Then there exists a collection of dyadic cubes $\left\{ Q^k_j \right\}_{j,k \in \nat}$ such that
\begin{itemize}
\item[{\rm{(i.)}}] for each $k,j \in \nat$, we have $Q^k_j \subset Q$ 
\item[{\rm{(ii.)}}] for almost every $x \in Q$, 
\begin{equation*}
| f(x) - m_{f}(Q) | \leq 4 M^{\sharp}_{2^{-n-2};Q}f(x) + 4 \displaystyle\sum_k \displaystyle\sum_{j} \omega_{2^{-n-2}} (f;Q^k_j) \mathbf{1}_{Q^k_j}(x)
\end{equation*}
\item[{\rm{(iii.)}}] for fixed $k$, $Q^k_j \cap Q^k_i = \emptyset $ for $i \neq j$ 
\item[{\rm{(iv.)}}] letting $\Omega_k = \displaystyle\bigcup_{j} Q^k_j$, we have $| \Omega_k \cap Q^k_j| \leq 2^{-1} |Q^k_j|$ and $\Omega_{k+1} \subset \Omega_{k}$.
\end{itemize}  
\end{theorem}
The second theorem is a result of E. Sawyer from \cite{sawyerm}:
\begin{theorem} \label{sawyer}
Let $w$ and $v$ be weights with $1 < p < \infty$ and define $\sigma=v^{1-p'}$. Then $M$ is bounded from $L^{p}(v)$ to $L^{p}(w)$ if and only if
\begin{equation}\label{e.testingM}
\int_{Q} \left |M(\sigma 1_{Q})(x)\right |^{p} v \lesssim \sigma(Q) \quad \text{for all } \,Q \text{ cubes}.
\end{equation} 
\end{theorem}
We finish this section with a definition
\begin{definition}
Let $\{ \mathcal Q_k \}_{k \in \mathbb{Z}}$ be collections of cubes as in Lemma \ref{whitney} and $R$ a dyadic cube. Provided there exists $k$ such that $R \in \mathcal Q_k$, define $C(R) = \sup \{ k : R \in \mathcal Q_k \}$, $c(R) = \inf \{ k : R \in \mathcal Q_k \}$ and $D(R) = C(R) - c(R)$; otherwise let $c(R) = C(R) = D(R) = 0$. 
\end{definition}
\section{Preliminary Results}
Here we formulate and prove some results which will be used in the argument for Theorem \ref{thm: t.mainch1}. We begin with the following weak-type estimate:
\begin{lemma} \label{weaktype}
Assuming $(\ref{testingo})$ and $(\ref{testingd})$ hold, for $\mathbf{g} \in L^{p^{\prime}}_{\ell^{r^{\prime}}}(w)$ and $f \in L^p(\sigma)$, we have
\begin{align}
\norme{\mathbf{U}_{\mathcal Q} ( \mathbf{g} w )}_{L^{p^{\prime}, \infty} (\sigma)} &\lesssim \mathcal L_{\ast}^{\frac{1}{p}} \norme{ \mathbf{g} }_{L^{p^{\prime}}_{\ell^{r^{\prime}}}(w)}, \label{weakint} 
\\
\norme{\tvec( f \sigma)}_{L^{p, \infty}(w)} &\lesssim \mathcal L^{\frac{1}{p^{\prime}}} \norme{f}_{L^p(\sigma)}. \label{oint}
\end{align}
\end{lemma}
A consequence of Lemma \ref{weaktype} is that we can make slight modifications to the testing conditions on $\tvec$ and $\mathbf{U}_{\mathcal Q}$:
\begin{lemma} \label{weakvec} For each $Q \in \mathcal D$ and for any positive $\mathbf{a} = \left\{a_I\right\}_{I \in \mathcal Q }$ satisfying $\sum_{I \in \mathcal Q} |a_I(x) |^r = 1$ for almost all $x \in \mathbb{R}^n$, we have 
\begin{align}
\nint_{\ral^n} \tvec (\mathbf{1}_Q \sigma)(x)^p w &\lesssim \mathcal{L} \sigma(Q),  \label{testt} \\
\nint_{\ral^n} \mathbf{U}_{\mathcal Q} \left( \mathbf{1}_{Q} \mathbf{a}  w  \right)(x)^{p^{\prime}} \sigma &\lesssim \mathcal{L}_{\ast} w(Q) . \label{testu}
\end{align}
\end{lemma}

Now we consider the following the lemma: 
\begin{lemma} \label{max}
Given collections of cubes $\left\{ \mathcal Q_k \right\}_{k \in \zat}$ as in Lemma \ref{whitney},  for each $k$ and $ Q\in \mathcal Q_k$ we have 
\begin{equation*}
 \max \left\{ 
  \mathcal{T}_{\mathcal Q, r, Q^{(1)}} ^{\textup{out}} ( \mathbf 1_{Q ^{(2)}} f \sigma ) (x)
 \,,\,  \mathcal{T}_{\mathcal Q,r}  (\mathbf 1_{(Q^{(2)}) ^{c}} f \sigma ) (x) 
\right\} \le 2 ^{k} 
\,, 
\end{equation*} with $x \in Q$. 
\end{lemma}
\noindent Further, Lemma \ref{max} also implies the following maximum principle
\begin{lemma} \label{l.max} For a given function $f \in L^1_{\rm{loc}}(\mathbb{R}^n)$, let $\Omega_k = \{ x \in \mathbb{R}^n: \mathcal{T}_{\mathcal Q,r}f(x) > 2^k \}$. Denote by $\mathcal Q_k$ the corresponding Whitney cubes for the $\Omega_k$ 
and for a given cube $Q$ let
\begin{equation*}
E_k(Q) = Q \cap (\Omega _{k+2} - \Omega _{k+3})\,, \qquad Q\in \mathcal Q_k \,. 
\end{equation*} Then for all $k$ and $x \in E_k(Q)$, we have 
\begin{equation*}
 2^k \leq \mathcal{T} ^{\textup{in}}_{\mathcal Q, r, Q^{(1)}} (\mathbf{1}_{Q^{(1)}} f)(x) . \end{equation*}
\end{lemma}

\subsection{Proof of Lemma \ref{weaktype}}
We will argue the case for $(\ref{weakint})$ first. 
Fix a sequence $\mathbf{g} \in L^{p^{\prime}}_{\ell^{r^{\prime}}}(w)$ and begin by defining $\Gamma_{\alpha} = \left\{ x : \mathbf{ U } ( \mathbf{g} w )(x) > \alpha\right\}$ for $\alpha > 0$. $\mathbf{U}_{\mathcal Q} ( \mathbf{g} w)(x)$ is
lower semi-continuous and so $\Gamma_{\alpha}$ is open. Similar to Lemma \ref{whitney}, we will perform a Whitney-style decomposition; specifically, for fixed $\alpha$, let $\left\{ L^{\alpha}_j \right\}_{j
\in \nat}$ be the dyadic cubes which are maximal with respect to the following two conditions: (i.) $L^{\alpha}_j \cap \Gamma_{2\alpha} \neq \emptyset$ and (ii.) $ L^{\alpha}_j \subset \Gamma_{\alpha}$ for all $ j \in \nat$.
First, we aim to put ourselves in a position to use the testing condition on $\tvec$; for fixed $j$,
\begin{align*}
\int _{L^{\alpha}_j} \mathbf{ U}_{\mathcal Q} (\mathbf{g} w)(x)  \sigma   
&= \nint _{L^{\alpha}_j} \langle \mathbf{1}_{L^{\alpha}_j} \sigma , \mathbf{U}_{\mathcal Q} ( \mathbf{g} w ) \rangle_{\ell^r} dx \\
&= \int _{L^{\alpha}_j} \langle  \mathbf{T}_{\mathcal Q,r}( \mathbf{1}_{L^{\alpha}_j} \sigma) , \mathbf{g} w \rangle_{\ell^r} dx \\
&\le \int_{L^{\alpha}_j} \tvec(\mathbf{1}_{L^{\alpha}_j} \sigma)(x) \lVert \mathbf{g} \rVert_{\ell^r} w.
\end{align*}
Now as a result, we have
\begin{align*}
\left( \sigma(L^{\alpha}_j)^{-1} \int _{L^{\alpha}_j} \mathbf{ U}_{\mathcal Q} (\mathbf{g} w)(x) \sigma \right)^{p'} 
&\le
\left( \sigma( L^{\alpha}_j)^{-1} \int _{L^{\alpha}_j} \tvec( \mathbf{1}_{L^{\alpha}_j} \sigma)(x) \norme{ \mathbf{g}}_{\ell^{r^{\prime}}}  w \right)^{p^{\prime}} 
\notag  \\
&\le 
\left(  \int_{L^{\alpha}_j} \norme{\mathbf{g}}_{\ell^{r^{\prime}}}^{p^{\prime}}  w  \right) 
\left( \int _{L^{\alpha}_j} \tvec( \mathbf{1}_{L^{\alpha}_j} \sigma)^p w \right)^{\frac{p^{\prime}}{p}}  \sigma(L^{\alpha}_j)^{- p^{\prime}}
 \notag \\  
&\lesssim \mathcal L^{\frac{p^{\prime}}{p}} \left(
 \int_{L^{\alpha}_j} \norme{\mathbf{g} }_{\ell^{r'}}^{p^{\prime}} w \right) \sigma(L^{\alpha}_j)^{\frac{p'}{p}-p'} \\
&=  \mathcal L^{\frac{p^{\prime}}{p}} \left(
 \int_{L^{\alpha}_j} \norme{\mathbf{g} }_{\ell^{r'}}^{p^{\prime}} w \right) \sigma(L^{\alpha}_j)^{-1} \notag .
\end{align*}
\noindent As a consequence,
\begin{equation*}
\left( \sigma(L^{\alpha}_j)^{-1} \int _{L^{\alpha}_j} \mathbf{ U}_{\mathcal Q} (\mathbf{g} w)(x) \sigma \right)^{p'} \sigma(L^{\alpha}_j)
\lesssim \mathcal L^{\frac{p^{\prime}}{p}} \left(
 \int_{L^{\alpha}_j} \norme{\mathbf{g} }_{\ell^{r'}}^{p^{\prime}} w \right)
\end{equation*} 
\noindent and summing over $j$ gives
\begin{align}
\sum_{j \in \mathbb{N}} \left( \sigma ( L^{\alpha}_j)^{-1} \int _{L^{\alpha}_j} \mathbf{ U} (\mathbf{g} w)(x)  \sigma \right)^{p^{\prime}} \sigma(L^{\alpha}_j) &\lesssim \mathcal L^{\frac{p^{\prime}}{p}} \norme{\mathbf{g} }_{L^{p^{\prime}}_{\ell^{r^{\prime}}}(w)}^{p^{\prime}}  \label{e.needthis}.
\end{align} 
\indent At this point we will appeal to a `good-lambda' trick. In particular, we fix $\alpha$ and $\epsilon = 2^{-p^{\prime} - 1}> 0$; further, we define $\mathcal E = \left\{ j : \sigma ( L^{\alpha}_j
\cap \Gamma_{2 \alpha} ) <
\epsilon \sigma(L^{\alpha}_j) \right\}$. So,
\begin{align*}
(2 \alpha)^{p^{\prime}} \sigma( \Gamma_{2\alpha} ) &\lesssim \epsilon (2 \alpha)^{p^{\prime}} \displaystyle\sum_{ j \in \mathcal E} \sigma(L^{\alpha}_j) + \epsilon^{-1} \displaystyle\sum_{j \not\in \mathcal E} (2 \alpha)^{p^{\prime}} \sigma(L^{\alpha}_j) \\ &\le 
 \epsilon (2 \alpha)^{p^{\prime}} \sum_{ j \in \mathcal E} \sigma(L^{\alpha}_j) +  \sum_{j \not\in \mathcal E} 2^{-1} (\alpha \sigma(L^{\alpha}_j) \sigma(L^{\alpha}_j)^{-1} )^{p^{\prime}} \sigma(L^{\alpha}_j) \\
&\leq 
 \epsilon (2 \alpha)^{p^{\prime}} \sum_{ j \in \mathcal E} \sigma(L^{\alpha}_j) + \sum_{j \not\in \mathcal E} 2^{-1} \left( \sigma (L^{\alpha}_j)^{-1} \int_{L^{\alpha}_j} \mathbf{U}_{\mathcal Q}( \mathbf{g} w)(x)  \sigma \right)^{p^{\prime}}  \sigma(L^{\alpha}_j )\\ &\lesssim 
\epsilon (2 \alpha )^{p^{\prime}} \sum_{j \in \mathcal E} \sigma(L^{\alpha}_j)
+ 2^{-1} \mathcal L^{\frac{p^{\prime}}{p}} \norme{\mathbf{g} }_{L^{p^{\prime}}_{\ell^{r^{\prime}}}(w)}^{p^{\prime}} \ 
\end{align*} where the final inequality follows from (\ref{e.needthis}). 
Hence
\begin{align*} (2 \alpha )^{p^{\prime}} \sigma( \Gamma_{2 \alpha} ) &\lesssim  2^{-1} (\alpha )^{p^{\prime}} \sigma( \Gamma_{\alpha}
) + 2^{-1} \mathcal L^{\frac{p^{\prime}}{p}} \norme{ \mathbf{g} }_{L^{p^{\prime}}_{\ell^{r^{\prime}}} (w)}^{p^{\prime}}
\\
&\le 2^{-1} \norme{ \mathbf{U}_{\mathcal Q}( \mathbf{g} w )}_{L^{p^{\prime}, \infty}(\sigma)}^{p^{\prime}} + 2^{-1} \mathcal L^{\frac{p^{\prime}}{p}} \norme{\mathbf{g} }_{L^{p^{\prime}}_{\ell^{r^{\prime}}}(w)}^{p^{\prime}}
\end{align*}
which gives $(\ref{weakint})$. \\
\indent Now we consider $(\ref{oint})$. The argument will be similar to that for $(\ref{weakint})$. Fix a positive function $f \in L^p(\sigma)$ and let $\Psi_{\alpha} = \left\{ x : \tvec (f \sigma)(x) > \alpha \right\}$ for $\alpha > 0$. Again, we perform a Whitney-style decomposition; explicitly, let $\left\{ P^{\alpha}_j \right\}_{j \in \nat}$ be the dyadic cubes which are maximal with respect to: (i.) $P^{\alpha}_j \cap \Psi_{2 \alpha} \neq \emptyset$ and (ii.) $P^{\alpha}_j \subset \Psi_{\alpha} $ for all $j \in \nat$. 
We define $\mathbf{a} = \mathbf{T}_{\mathcal Q,r}(f \sigma)^{r-1} (\tvec(f \sigma))^{-1}$ and attempt to place ourselves in a position where we may use the testing condition on $\mathbf{U}_{\mathcal Q}$; using duality as before, for each $j$ we see the expression
\begin{equation}
\left( w(P^{\alpha}_j)^{-1} \int_{P^{\alpha}_j} \tvec (f \sigma)(x) w \right)^{p} w(P^{\alpha}_j) 
\end{equation}
\noindent is equivalent to
\begin{equation}
\left( w(P^{\alpha}_j)^{-1} \int_{P^{\alpha}_j} \mathbf{U}_{\mathcal Q}( \mathbf{1}_{P^{\alpha}_j} w \mathbf{a} )(x) f(x) \sigma \right)^p w(P^{\alpha}_j). \label{testingu}
\end{equation}
\noindent Using H\"{o}lder's inequality,
\begin{align*}
(\ref{testingu}) &\le \left( \int_{P^{\alpha}_j} \mathbf{U}_{\mathcal Q} (\mathbf{ a} \mathbf{1}_{P^{\alpha}_j} w)(x) ^{p^{\prime}} \sigma \right)^{\frac{p}{p^{\prime}}} \left( \int_{P^{\alpha}_j} f(x)^p \sigma \right) w(P^{\alpha}_j)^{1 - p} \\
&\lesssim \mathcal L_{\ast}^{\frac{p}{p^{\prime}}} \left( \int_{P^{\alpha}_j} f(x)^p \sigma \right) 
\end{align*}
\noindent and summing gives
\begin{align*}
\sum_{j \in \mathbb{N}} \left( w(P^{\alpha}_j)^{-1} \int_{P^{\alpha}_j} \tvec (f \sigma)(x) w \right)^{p} w(P^{\alpha}_j)
&\lesssim \mathcal L_{\ast}^{\frac{p}{p^{\prime}}} \norme{f}^p_{L^p(\sigma)}.
\end{align*}
As before we use a `good-lambda' trick; we fix $\alpha$ and $\epsilon = 2^{-p-1}$. Further, define $\Upsilon = \left\{ j : w( P^{\alpha}_j \cap \Psi_{2 \alpha} ) < \epsilon w(P^{\alpha}_j) \right\}$. So 
\begin{align*}
(2 \alpha)^p w( \Psi_{2 \alpha}) &\lesssim \epsilon (2 \alpha)^p \sum_{j \in \Upsilon} w( P^{\alpha}_j) + \epsilon^{-1} \sum_{j \not\in \Upsilon} (2 \alpha)^p w(P^{\alpha}_j) \\
&\lesssim \epsilon (2 \alpha)^p \sum_{j \in \Upsilon} w(P^{\alpha}_j) + 2^{-1} \sum_{j \not\in \Upsilon} (\alpha w(P^{\alpha}_j )w(P^{\alpha}_j)^{-1} )^p w(P^{\alpha}_j) \\
&\lesssim \epsilon (2 \alpha)^p \sum_{j \in \Upsilon} w(P^{\alpha}_j) + 2^{-1} \sum_{j \not\in \Upsilon} \left( w (P^{\alpha}_j)^{-1} \int_{P^{\alpha}_j} \tvec (f \sigma)(x) w \right)^p w(P^{\alpha}_j) \\
&\lesssim \epsilon (2 \alpha)^p \sum_{j \in \Upsilon} w(P^{\alpha}_j) + 2^{-1} \mathcal L_{\ast}^{\frac{p}{p^{\prime}}} \norme{f}_{L^p(\sigma)}^p \\
&\leq
\epsilon (2 \alpha)^p w(\Psi_{\alpha}) + 2^{-1} \mathcal L_{\ast}^{\frac{p}{p^{\prime}}} \norme{f}_{L^p(\sigma)}^p.
\end{align*}
Now we have
\begin{align*}
(2 \alpha)^p w( \Psi_{2 \alpha}) &\lesssim 2^{-1} \alpha^p w( \Psi_{\alpha}) + 2^{-1} \mathcal L_{\ast}^{\frac{p}{p^{\prime}}} \norme{f}_{L^p(\sigma)}^p \\
&\le 2^{-1} \norme{\tvec{(f \sigma)}}_{L^{p,\infty}(w)}^{p} +  \mathcal L_{\ast} ^{\frac{p}{p^{\prime}}} \norme{f}_{L^p(\sigma)}^p  \end{align*}
and this gives $(\ref{oint})$.

\subsection{Proof of Lemma \ref{weakvec}}
First, we will show the case for $(\ref{testt})$. By (\ref{weakint}) and duality,
 we have for each $f \in L^{p, 1}(\sigma)$, 
\begin{eqnarray*}
\norme{\tvec (f \sigma)}_{L^p(w)} \lesssim \mathcal L^{\frac{1}{p}} \norme{f}_{L^{p,1}(\sigma)} .
\end{eqnarray*}
Since for any cube $Q$, $\mathbf{1}_{Q} \in L^{p,1}(\sigma)$ and 
$\norme{\mathbf{1}_{Q}}_{L^{p,1}(\sigma)} = \sigma(Q)^{\frac{1}{p}}$, we have
\[ \norme{ \tvec ({ \mathbf{1}_{Q} \sigma})}_{L^p(w)} \lesssim \mathcal L^{\frac{1}{p}} \sigma(Q)^{\frac{1}{p}} \] 
which gives the desired result. \\
\indent We conclude by verifying $(\ref{testu})$ holds. Consider, for $\mathbf{a} = \mathbf{T}_{\mathcal Q,r}(f \sigma) \mathcal{T}_{\mathcal Q,r}(f \sigma)^{-1}$ and $Q$ fixed,
\begin{align*}
\left(\int_{\ral^n} \mathbf{U}_{\mathcal Q} (\mathbf{1}_{Q} \mathbf{a} w )(x)^{p^{\prime}} \sigma \right)^{\frac{1}{p^{\prime}}} &= \int_{\ral^n} \mathbf{U}_{\mathcal Q} ( \mathbf{1}_{Q} \mathbf{a} w)(x) h(x) \sigma
\end{align*}
for some $h \in L^{p}(\sigma)$. Then using duality and H\"older's inequality in $\ell^r-\ell^{r^{\prime}}$ we have
\begin{align}
\int_{\ral^n} \mathbf{U}_{\mathcal Q} ( \mathbf{1}_{Q} \mathbf{a} w)(x) h(x) \sigma &= 
\int_{\ral^n} \langle \mathbf{1}_{Q} \mathbf{ a} w , \mathbf{T}_{\mathcal Q,r}(h \sigma) \rangle_{\ell^r} dx \notag \\
&\leq \int_{Q} \tvec(h \sigma)(x) w . \label{no} 
\end{align}
Recall, by $(\tvec(h \sigma)(x))^{\ast}$ and $(\mathbf{1}_{Q})(x))^{\ast}$, we mean the symmetric decreasing rearrangements of $\tvec(h \sigma)(x)$ and $\mathbf{1}_{Q}(x)$ with respect to $w$. We continue from $(\ref{no})$ by applying H\"older's inequality and using $(\ref{oint})$ to obtain
\begin{align}
(\ref{no}) 
&\le \nint_{\ral} (\tvec(h \sigma)(x))^{\ast} \left( \mathbf{1}_{Q}(x) \right)^{\ast} w \notag \\
&\le \norme{\tvec(h\sigma)}_{L^{p,\infty}(w)} w(Q)^{\frac{1}{p^{\prime}}} \notag \\
&\leq \norme{\tvec(\cdot \sigma)}_{L^{p}(\sigma) \rt L^{p, \infty}(w)}  w(Q)^{\frac{1}{p^{\prime}}} \notag \\
&\lesssim \mathcal L_{\ast}^{\frac{1}{p^{\prime}}}  w(Q)^{\frac{1}{p^{\prime}}} . \notag
\end{align}
The foregoing inequalities yield 
\begin{equation*} \int_{\ral^n} \mathbf{U}_{\mathcal Q}( \mathbf{1}_{Q} \mathbf{a} w)(x)^{p^{\prime}} \sigma \leq \mathcal L_{\ast} w(Q)^{\frac{p^{\prime}}{p^{\prime}}} \end{equation*}
and we are done.  

\subsection{Proof of Lemma \ref{max} and Lemma \ref{l.max}}
\subsubsection{Proof of Lemma \ref{max}}
By Lemma \ref{whitney}, there is 
 $ z\in Q^{(2)} \cap \Omega _{k} ^{c}$. Thus for $ x\in Q$ we have
\begin{equation*}
 \tvec  (\mathbf 1_{(Q^{(2)}) ^{c}} f \sigma ) (x) 
=   \mathcal{T}_{\mathcal Q,r,Q^{(1)}} ^{\textup{out}} (\mathbf 1_{(Q^{(2)}) ^{c}} f \sigma ) (x)  
\le  \tvec (f \sigma ) (z) \le 
2 ^{k}
\end{equation*}
and we are done.
\subsubsection{Proof of Lemma \ref{l.max}}
By Lemma \ref{max} and the sub-linearity of $\mathcal{T}_{\mathcal Q}$, we have for $ x\in E_k(Q)$
\begin{eqnarray*}
2 ^{k+2} - 2 ^{k+1} &\leq& \tvec  (f) (x) - 
  \mathcal{T}_{\mathcal Q, r,{Q^{(1)}}} ^{\textup{out}} ( \mathbf 1_{Q ^{(1)}} f) (x)
 -  \tvec  (\mathbf 1_{(Q^{(1)}) ^{c}} f \sigma ) (x) \\
 &\leq& 
 \mathcal{T}_{\mathcal Q, r,{Q^{(1)}}}  ^{\textup{in}}  (\mathbf 1_{Q^{(1)}} f) (x) .
\end{eqnarray*}
Noting $2^{k+2} - 2^{k+1} \ge 2^k$, we obtain $2^k \leq  \mathcal{T}_{\mathcal Q, r,{Q^{(1)}}} ^{\textup{in}}  (\mathbf 1_{Q^{(1)}} f) (x)$. 
\subsubsection{Proof of Corollary \ref{c.twtsq}}
The following lemma is known (see \cite{cump} and \cite{lernerlps}):
\begin{lemma}
Let $f \in L^1_{\rm{loc}}(\mathbb{R}^n)$ and $\mathbf{g}$ be a sequence of $\ell^r$ summable locally integrable functions. For $\mathbf{M}_r$ the vector-valued maximal function with exponent $r$ and $S$ the dyadic square function,
\begin{align*}
\omega_{\lambda}(Sf^2,Q_0) &\lesssim \lambda^{-1} \langle f \rangle_{\rho Q_0} ^2 \\
\omega_{\lambda}(\mathbf{M}_r(\mathbf{g})^r,Q_0) &\lesssim \lambda^{-1} \langle \lVert \mathbf{g} \rVert_{\ell^r} \rangle_{Q_0}^r.
\end{align*}
\end{lemma}
\noindent By Lerner's decomposition theorem, for each cube $Q_N$ there is an appropriate collection of sparse cubes $\mathcal Q_N$ and $\mathcal I_N$ such that
\begin{align*}
| Sf(x) - m_{Q_N} | &\lesssim M^{\sharp}(f)(x) + \mathcal {T}_{\mathcal I_N, \rho, r}(f)(x) \\
| \mathbf{M}_r(\mathbf{g})(x) - m_{Q_N} | &\lesssim M^{\sharp}(\lVert \mathbf{g} \rVert_{\ell^r})(x) + \mathcal {T}_{\mathcal Q_N, r}(\lVert \mathbf{g} \rVert_{\ell^r})(x). 
\end{align*}
\noindent Assuming Theorem \ref{thm: t.mainch1} and recalling Sawyer's two weight theorem for the maximal function, the conclusion of the corollary for $\mathbf{M}_r$ and the dyadic square function $S$ follow immediately. 
\subsection{Proof of Theorem \ref{thm: t.mainch1}: Necessity}
\indent Here we prove the necessity of the testing conditions. We suppose that $\tvec$ is a bounded operator. The necessity of $(\ref{testingo})$ is immediate by taking $f = \mathbf{1}_{Q}$ for an arbitrary cube, so we only need to verify the necessity of the conditions on $\mathbf{U}_{\mathcal Q}$. Fix a cube $Q$ and a sequence $\mathbf{a}$ such that $\lVert \mathbf{a} \rVert_{\ell^r} = 1$. Without loss of generality we assume $h$ and $\mathbf{a}$ are positive. Then,
\begin{align*} 
\left( \int_{Q} \mathbf{U}_{\mathcal Q} ( \mathbf{ a} \mathbf{1}_Q w)(x)^{p'} \sigma \right)^{\frac{1}{p'}} &=
\int_{Q} \mathbf{U}_{\mathcal Q}(\mathbf{a} \mathbf{1}_{Q} w)(x) h \sigma
\end{align*}
where $h$ is an appropriate function from $L^p(\sigma)$ satisfying $\norme{h}_{L^{p}(\sigma)} = 1$.
Now we use duality and apply H\"{o}lder's inequality in $\ell^{r}-\ell^{r^{\prime}}$ and obtain
\begin{align*}
\int_{Q} \mathbf{U}_{\mathcal Q}(\mathbf{a} \mathbf{1}_{Q} w)(x) h \sigma &=
\int_{\mathbb{R}^n} \langle \mathbf{T}_{\mathcal Q,r}(h \mathbf{1}_{Q} \sigma), \mathbf{a} \mathbf{1}_{Q} w \rangle_{\ell^r} dx \\ 
&\leq  
\int_{Q} \tvec ( h \mathbf{1}_{Q} \sigma)(x)  w  \\ &\leq 
\norme{\tvec( h \mathbf{1}_{Q} \sigma)}_{L^p(w)} w(Q)^{\frac{1}{p^{\prime}}} \\ & \leq \norme{\tvec(\cdot \sigma)}_{L^p(\sigma) \rt L^p(w)} w(Q)^{\frac{1}{p^{\prime}}}. 
\end{align*}
Hence,
\begin{equation*}
 \int_{Q}  \mathbf{U}_{\mathcal Q} ( \mathbf{ a} \mathbf{1}_{Q} w )(x)^{{p^{\prime}}} \sigma   \le \norme{\tvec(\cdot \sigma)}^{p^{\prime}}_{L^p(\sigma) \rt L^p(w)} w(Q)
\end{equation*}
where $\mathbf{a}$ is arbitrary. Taking supremums we have
\begin{align*}
\sup_{ \mathbf{a} } \sup_{ Q} w(Q)^{-1} \int_{Q} \mathbf{U}_{\mathcal Q} (  \mathbf{1}_{Q} \mathbf{a}  w )(x)^{p^{\prime}}\sigma
\le \norme{\tvec(\cdot \sigma)}^{p^{\prime}}_{L^p(\sigma) \rt L^p(w)} \label{testingsup}
\end{align*}
which gives the result.

\section{Proof of Theorem \ref{thm: t.mainch1}: Sufficiency} 
We apply Lemma \ref{whitney} to obtain a collection of cubes $\mathcal Q_k$ for each $k$ such that \begin{align*} \Omega_k &=  \left\{ x \in \mathbb{R}^n: \mathcal{T}_{\mathcal Q,r}(f\sigma)(x) > 2^k \right\} \\ &= \cup_{Q \in \mathcal Q_k} Q. \end{align*} For $Q \in \mathcal Q_k$, define $E_k(Q) = (\Omega_k \backslash \Omega_{k+2}) \cap Q$. Then we have the following:
\begin{align*}
\int_{\mathbb{R}^n}  \mathcal{T}_{\mathcal{Q},r}(f \sigma)(x)^p w &\lesssim
\sum_{k \in \mathbb{Z}} w(\left\{ x\in \mathbb{R}^n : \mathcal{T}_{\mathcal{Q},r}(f \sigma)(x) > 2^k  \right\}) 2^{kp} \\
&\lesssim \sum_{k \in \mathbb{Z}} \sum_{Q \in \mathcal Q_k} w(E_k(Q)) 2^{kp}.
\end{align*} 
By Lemma \ref{l.max}, 
\begin{align*}
w(E_k(Q)) 2^k &\lesssim \int_{E_k(Q)} \mathcal{T}_{\mathcal Q,r}(f \sigma \mathbf{1}_{Q^{(1)}})(x) w \\
&= \int_{Q^{(1)}} \mathbf{U}_{\mathcal Q}(\mathbf{a} \mathbf{1}_{E_k(Q)} w)(x) f(x) \sigma;
\end{align*}
we split the above integral into two pieces so that 
\begin{align*}
\int_{Q^{(1)}} \mathbf{U}_{\mathcal Q}(\mathbf{a} \mathbf{1}_{E_k(Q)} w)(x) f(x) \sigma &=
S_{1,k}(Q) + S_{2,k}(Q)
\end{align*}
with
\begin{align*}
S_{1,k}(Q) &=\int_{Q^{(1)} \backslash \Omega_{k+m}} \mathbf{U}_{\mathcal Q}(\mathbf{a} \mathbf{1}_{E_k(Q)} w)(x) f(x) \sigma \\
S_{2,k}(Q) &= \int_{Q^{(1)} \cap \Omega_{k+m}} \mathbf{U}_{\mathcal Q}(\mathbf{a} \mathbf{1}_{E_k(Q)})(x) f(x) \sigma. 
\end{align*}
For each $k$, we partition $\mathcal Q_k$ into two collections:
\begin{align*}
\mathcal Q_{1,k} &= \left\{ Q \in \mathcal Q_k :  w(E_k(Q)) \le \eta w(Q) \right\} \\
\mathcal Q_{2,k} &= \left\{ Q \in \mathcal Q_k:  w(E_k(Q)) > \eta w(Q) \right\} \\
\end{align*}
where $0 < \eta < 1$ is a fixed parameter that will be defined later in the proof; further divide $\mathcal Q_{2,k}$ into:
\begin{align*}
\mathcal Q^2_k &= \left\{ Q \in \mathcal Q_{2,k} : S_{2,k}(Q) \le S_{1,k}(Q) \right\} \\
\mathcal Q^3_k &= \left\{ Q \in \mathcal Q_{2,k} : S_{2,k}(Q) > S_{1,k}(Q) \right\}.
\end{align*}
The sum $\sum_{k \in \mathbb{Z}} \sum_{Q \in \mathcal Q_k} w(E_k(Q)) 2^{kp}$ is split into pieces corresponding to the collections above:
\begin{align*}
I_1 &= \sum_{k \in \mathbb{Z}} \sum_{Q \in \mathcal Q_{1,k}} w(E_k(Q)) 2^{kp} \\
I_2 &= \sum_{k \in \mathbb{Z}} \sum_{Q \in \mathcal Q^2_k} w(E_k(Q)) 2^{kp} \\
I_3 &= \sum_{k \in \mathbb{Z}} \sum_{Q \in \mathcal Q^3_k} w(E_k(Q)) 2^{kp}.
\end{align*} 
Trivially, we have 
\begin{align*}
\sum_{k \in \mathbb{Z}} \sum_{Q \in \mathcal Q_k} w(E_k(Q)) 2^{kp} &=
I_1 + I_2 + I_3
\end{align*}
so that it suffices to estimate each $I_j$. 
\subsection{Estimating $I_1$}
Consider,
\begin{align*}
I_1 &\lesssim \eta \sum_{k \in \mathbb{Z}} \sum_{Q \in \mathcal Q_k} \eta^{-1} w(Q) 2^{kp} \\
&\lesssim \eta \int_{\mathbb{R}^n} \mathcal T_{\mathcal Q, r}(f \sigma)(x)^p w;
\end{align*}
as $0 < \eta < 1$, we may absorb the term $I_1$ into $\lVert \mathcal T_{\mathcal Q, r} (f \sigma) \rVert_{L^p(\sigma)}$. 
\subsection{Estimating $I_2$}
Here, notice
\begin{align*}
\eta 2^k w(Q) &\le \int_{Q^{(1)}} \mathbf{U}_{\mathcal Q}(\mathbf{a} \mathbf{1}_{E_k(Q)}w)(x) f(x) \sigma \\
&\lesssim \int_{Q^{(1)} \backslash \Omega_{k+m}} \mathbf{U}_{\mathcal Q}(\mathbf{a} \mathbf{1}_{E_k(Q)}w)(x) f(x) \sigma  \\
&\le \left( \int_{Q^{(1)} \backslash \Omega_{k+m}} \mathbf{U}_{\mathcal Q}(\mathbf{a} \mathbf{1}_{Q} w)(x)^{p'} \sigma \right)^{\frac{1}{p'}} \left( \int_{Q^{(1)} \backslash \Omega_{k+m}} f(x)^p \sigma \right)^{\frac{1}{p}} \\
&\lesssim \mathcal L_{\ast}^{\frac{1}{p'}}
w(Q)^{\frac{1}{p'}} \left( \int_{Q^{(1)} \backslash \Omega_{k+m}} f(x)^p \sigma \right)^{\frac{1}{p}}
\end{align*}
so that for fixed $Q$ and $k$,
\begin{align*}
w(E_k(Q)) 2^{kp} &\lesssim \eta^{-p} w(E_k(Q)) \left( \int_{Q^{(1)} \backslash \Omega_{k+m}} \mathbf{U}_{\mathcal Q}(\mathbf{a} \mathbf{1}_{E_k(Q)} w)(x) f(x) \sigma \right)^p  \\
&\lesssim \eta^{-p} w(E_k(Q)) \mathcal{L}_{\ast}^{\frac{p}{p'}}
w(Q)  \int_{Q^{(1)} \backslash \Omega_{k+m}} f(x)^p \sigma  \\
&= \eta^{-p} \mathcal{L}_{\ast}^{\frac{p}{p'}} \frac{w(E_k(Q))}{w(Q)} \int_{Q^{(1)} \backslash \Omega_{k+m}} f(x)^p \sigma \\
&\le \eta^{-p}  \mathcal{L}_{\ast}^{\frac{p}{p'}} \int_{Q^{(1)} \backslash \Omega_{k+m}} f(x)^p \sigma.
\end{align*}
Summing, we have from (\ref{e.crowd})
\begin{equation*} 
\eta^{-p} \mathcal{L}_{\ast}^{\frac{p}{p'}} \sum_{k \in \mathbb{Z}} \sum_{ Q \in \mathcal Q^2_k} \int_{Q^{(1)} \backslash \Omega_{k+m} } f(x)^p \sigma 
\lesssim \eta^{-p} \mathcal{L}_{\ast}^{\frac{p}{p'}} \int_{\mathbb{R}^n} f(x)^p \sigma;
\end{equation*}
\noindent recalling
\begin{align*}
I_2 &= \sum_{k \in \mathbb{Z}} \sum_{Q \in \mathcal Q^2_k} w(E_k(Q)) 2^{kp} \\
&\lesssim \eta^{-p} \sum_{k \in \mathbb{Z}} \sum_{ Q \in \mathcal Q^2_k} \int_{Q^{(1)} \backslash \Omega_{k+m} } f(x)^p \sigma,
\end{align*}
\noindent implies the result.
\subsection{Estimating $I_3$}
Assume $N$ is some fixed positive integer and $0 \le n < m$; we split the remaining cubes into collections modulo $m$ and intend to show
\begin{align*}
\sum_{\substack{ k > -N \\ k \equiv n \mod{m}} } \sum_{Q \in \mathcal Q^3_k} w(E_k(Q)) 2^{kp} &\lesssim \int_{\mathbb{R}^n} f(x)^p \sigma
\end{align*}
with implied constants independent of $n$ and $N$. The monotone convergence theorem combined with summing over $n$ will yield
\begin{align*}
\sum_{Q \in \mathcal Q^3_k} w(E_k(Q)) 2^{kp} &\lesssim \int_{\mathbb{R}^n} f(x)^p \sigma
\end{align*}
\noindent To this end, we use a stopping time argument. Namely, set $\mathcal P(N,n,1)$ to be the collection of maximal cubes within $P_{N,n} = \cup_{\substack{j \equiv n \mod m \\ j \ge -N}} \cup_{Q \in \mathcal Q^3_j} Q$. For $j > 1$ define $\mathcal P(N,n,j)$ to be the collection of all cubes $I$ in $ P_{N,n}$ which satisfy the following:
\begin{itemize}
\item[(i.)] there is $I' \in \mathcal P(N,n,j-1)$ such that $I \subsetneq I'$ 
\item[(ii.)] $\langle f \rangle^{\sigma}_I > 2 \langle f \rangle^{\sigma}_{I'}$ 
\item[(iii.)] $I$ is maximal with respect to properties (i.) and (ii.)
\end{itemize} 
\noindent Denote by $\mathcal P(N,n) = \cup_{j=1}^{\infty} \mathcal P(N,n,j)$.
\\
\indent We define for $Q \in \mathcal Q^3_k$
\begin{align*}
\mathcal N(k,m,N,n,Q) &= \left\{ I \in \mathcal Q_{k+m}, k \equiv n \mod{m} : I \cap Q^{(1)} \neq \emptyset \right\} \\ \mathcal N(k,m,N,n) &= \cup_{\substack{Q \in \mathcal Q_{k} \\ k \equiv n \mod{m}}} \mathcal N(k,m,N,n,Q)
\end{align*} 
and note that $Q^{(1)} \cap \Omega_{k+m} = \cup_{I \in \mathcal N(k,m,N,n,Q)} I$. Further, for each $I \in \mathcal N(k,m,N,n)$ there is $I_{k,m,N,n} \in \mathcal Q_k$ such that $I \subset I_{k,m,N,n}$. Since $k \equiv n \mod{m}$ we have $I \in \mathcal P$ or $\Gamma(I) = \Gamma(I_{k,m,N,n})$; as a consequence, we may split the sum 
\begin{align*}
\int_{Q^{(1)} \cap \Omega_{k+m}} \mathbf{U}_{\mathcal Q}(\mathbf{a} \mathbf{1}_{E_k(Q)}w)(x) f(x) \sigma &=
\sum_{I \in \mathcal N(k,m,N,n,Q)} \int_{I} \mathbf{U}_{\mathcal Q}(\mathbf{a} \mathbf{1}_{E_k(Q)} w)(x) f(x) \sigma 
\end{align*}
into two pieces:
\begin{align*}
A_{1} (k,m,N,n,Q) &= \sum_{\substack{I \in \mathcal N(k,m,N,n,Q) \\ I \in \mathcal P(N,n)}} \int_{I} \mathbf{U}_{\mathcal Q}(\mathbf{a} \mathbf{1}_{E_k(Q)} w)(x) f(x) \sigma \\
A_{2}(k,m,N,n,Q) &= \sum_{\substack{I \in \mathcal N(k,m,N,n,Q) \\ \Gamma(I) = \Gamma(I_{k,m,N,n})}} \int_{I} \mathbf{U}_{\mathcal Q}(\mathbf{a} \mathbf{1}_{E_k(Q)} w)(x) f(x) \sigma.
\end{align*}
For the remainder of the proof, we will assume $k \equiv n \mod{m}$ and suppress the notational dependence on $N$ and $n$ (e.g. we will write $A_1(k,m,Q)$ for $A_1(k,m,N,n,Q)$). 
Continuing, from the defining properties of $\mathcal Q^3_k$,
\begin{align*}
2^k w(Q) &\lesssim \eta^{-1} \int_{Q^{(1)} \cap \Omega_{k+m}} \mathbf{U}_{\mathcal Q}(\mathbf{a} \mathbf{1}_{E_k(Q)} w)(x) f(x) \sigma \\
&\lesssim \eta^{-1} A_1(k,m,Q) + \eta^{-1} A_2(k,m,Q)
\end{align*}
so that
\begin{align*}
2^{kp} w(E_k(Q)) &\lesssim  \frac{w(E_k(Q))}{\eta^p w(Q)^p}  A_1(k,m,Q)^p + \frac{w(E_k(Q)) }{\eta^p w(Q)^p} A_2(k,m,Q)^p.
\end{align*}
Recalling 
\begin{align*}
I_3 &= \sum_{k \in \mathbb{Z} } \sum_{Q \in \mathcal Q^3_k} w(E_k(Q)) 2^{kp} 
\end{align*}
we see it is enough to estimate
$I_{3,j} = \sum_{Q \in \mathcal Q^3_k} I_{3,j}(Q)$ for $j \in \left\{1,2\right\}$ and 
\begin{align*}
I_{3,1}(Q) &=  \frac{w(E_k(Q))}{w(Q)^p} A_1(k,m,Q)^p \\
I_{3,2}(Q) &= \frac{w(E_k(Q))}{w(Q)^p} A_2(k,m,Q)^p
\end{align*}
with $Q \in \mathcal Q^3_k$.
\subsubsection{Estimating $I_{3,1}$}
For a fixed cube $Q$ and $I \in \mathcal N(k,m,Q)$ we may write
\begin{align*}
\int_{I} \mathbf{U}_{\mathcal Q}(\mathbf{a} \mathbf{1}_{E_k(Q)} w)(x) f(x) \sigma &=
\int_{I} \mathbf{U}_{\mathcal Q}(\mathbf{a} \mathbf{1}_{E_k(Q)} w)(x) \langle f \rangle^{\sigma}_I \sigma
\end{align*}
since the 
expression $\mathbf{U}_{\mathcal Q}(\mathbf{a} \mathbf{1}_{E_k(Q)} w)(x)$ is constant for $x \in I$. Continuing, for $ G \in \mathcal P$,
\begin{align*}
\sum_{\substack{I \in \mathcal N(k,m,Q) \\ \Gamma(I) = \Gamma(I_{k,m})}} \int_{I} \mathbf{U}_{\mathcal Q}(\mathbf{a} \mathbf{1}_{E_k(Q)} w)(x) f(x) \sigma &=
\sum_{\substack{I \in \mathcal N(k,m,Q) \\ \Gamma(I) = \Gamma(I_{k,m})}} \int_{I} \mathbf{U}_{\mathcal Q}(\mathbf{a} \mathbf{1}_{E_k(Q)} w)(x) f(x) \sigma \\
&= \sum_{\substack{I \in \mathcal N(k,m,Q) \\ \Gamma(I) = \Gamma(I_{k,m})}} \int_{I} \mathbf{U}_{\mathcal Q}(\mathbf{a} \mathbf{1}_{E_k(Q)} w)(x) \langle f \rangle^{\sigma}_{I} \sigma \\
&\lesssim \langle f \rangle^{\sigma}_G \sum_{\substack{I \in \mathcal N(k,m,Q) \\ \Gamma(I) = \Gamma(I_{k,m})}} \int_{I} \mathbf{U}_{\mathcal Q}(\mathbf{a} \mathbf{1}_{E_k(Q)} w)(x) \sigma.
\end{align*} 
So for fixed $G \in \mathcal P$, using duality and H\"{o}lder's inequality we have
\begin{align*}
\frac{w(E_k(Q))}{w(Q)^p} A_1(k,m,Q)^p &\lesssim
\frac{w(E_k(Q))}{w(Q)^p} (\langle f \rangle^{\sigma}_{G} )^p \left( \sum_{\substack{I \in \mathcal N(k,m,Q) \\ \Gamma(I) = \Gamma(I_{k,m}) = G}} \int_{I} \mathbf{U}_{\mathcal Q}(\mathbf{a} \mathbf{1}_{E_k(Q)} w)(x) \sigma  \right)^p \\
&\le \frac{w(E_k(Q))}{w(Q)^p} (\langle f \rangle^{\sigma}_G )^p \left( \sum_{\substack{I \in \mathcal N(k,m,Q) \\ \Gamma(I) = \Gamma(I_{k,m}) = G}} \int_{G} \mathcal{T}_{\mathcal Q, r}(\mathbf{1}_{G} \sigma)(x) w \right)^p \\
&\le (\langle (f) \rangle^{\sigma}_{G})^p w(E_k(Q)) M_{w}(\mathcal{T}_{\mathcal Q, r}(\mathbf{1}_{G} \sigma ))(x)^p. 
\end{align*}
By the universal maximal estimate and the modified testing condition Lemma \ref{weakvec},
\begin{align*}
\sum_{k \in \mathbb{Z}} \sum_{Q \in \mathcal Q^3_k} w(E_k(Q)) M_{w}(\mathcal{T}_{\mathcal Q, r}(\mathbf{1}_{G} \sigma ))(x)^p
&\lesssim \int_{\mathbb{R}^n} M_{w}(\mathcal{T}_{\mathcal Q, r}(\mathbf{1}_{G} \sigma ))(x)^p w \\
&\lesssim \int_{\mathbb{R}^n} \mathcal{T}_{\mathcal Q,r}(\mathbf{1}_{G} \sigma)(x)^p w \\
&\lesssim \mathcal L \sigma(G).
\end{align*}
Hence,
\begin{align*}
\sum_{k \in \mathbb{Z}} \sum_{Q \in \mathcal Q^3_k} \frac{w(E_k(Q))}{w(Q)^p} \left( \sum_{\substack{I \in \mathcal N(k,m,Q) \\ \Gamma(I) = \Gamma(I_{k,m})}} \int_{I} \mathbf{U}_{\mathcal Q}(\mathbf{a} \mathbf{1}_{E_k(Q)} w)(x) f(x) \sigma \right)^p &\lesssim
\mathcal L\sum_{G \in \mathcal P} (\langle f \rangle^{\sigma}_G)^p \sigma(G) \\
&\lesssim \mathcal L \int_{\mathbb{R}^n} f(x)^p \sigma
\end{align*}
where in the last line we have used the Carleson embedding theorem.
\subsubsection{Estimating $I_{3,2}$}
We begin by noticing for fixed $Q$,
\begin{align*}
\frac{w(E_k(Q))}{w(Q)^p} A_2(k,m,Q)^p &=
\frac{w(E_k(Q))}{w(Q)^p} \Biggl( \sum_{\substack{I \in \mathcal N(k,m,Q) \\ I \in \mathcal P}} \frac{\sigma(I)^{\frac{1}{p}}}{\sigma(I)^{\frac{1}{p}}}  \int_{I} \mathbf{U}_{\mathcal Q}(\mathbf{a} \mathbf{1}_{E_k(Q)} w)(x) f(x) \sigma \Biggr)^p \\
&\le I_{4,1}(k,m,Q) I_{4,2}(k,m,Q)
\end{align*}  
where we define
\begin{align*}
I_{4,1}(k,m,Q) &= \frac{w(E_k(Q))}{w(Q)^p} \Biggl( \sum_{\substack{I \in \mathcal N(k,m,Q) \\ I \in \mathcal P}} \sigma(I)^{\frac{-p'}{p}} \left( \int_{I} \mathbf{U}_{\mathcal Q}(\mathbf{a} \mathbf{1}_{E_k(Q)} w)(x) \sigma \right)^{p'} \Biggr)^{\frac{p}{p'}} \\
I_{4,2}(k,m,Q)&=  \sum_{\substack{I \in \mathcal N(k,m,Q) \\ I \in \mathcal P}} \sigma(I) (\langle f \rangle^{\sigma}_G)^p .
\end{align*}
Notice for each $Q$ by H\"{o}lder's inequality,
\begin{align*}
\sigma(I)^{\frac{-p'}{p}} \left( \int_{I} \mathbf{U}_{\mathcal Q}(\mathbf{a} \mathbf{1}_{E_k(Q)} w)(x) \sigma \right)^{p'} 
&\le  \sigma(I)^{\frac{-p'}{p} + \frac{p'}{p}} \int_{I} \mathbf{U}_{\mathcal Q}(\mathbf{a} \mathbf{1}_{E_k(Q)} w)(x)^{p'} \sigma,
\end{align*}
so that
\begin{align*}
\sum_{\substack{I \in \mathcal N(k,m,Q) \\ I \in \mathcal P}} \sigma(I)^{\frac{-p'}{p}} \left( \int_{I} \mathbf{U}_{\mathcal Q}(\mathbf{a} \mathbf{1}_{E_k(Q)} w)(x) \sigma \right)^{p'} 
&\le  \sum_{\substack{I \in \mathcal N(k,m,Q) \\ I \in \mathcal P}} \int_{I} \mathbf{U}_{\mathcal Q}(\mathbf{a} \mathbf{1}_{E_k(Q)} w)(x)^{p'} \sigma \\
&\lesssim   \int_{\mathbb{R}^n} \mathbf{U}_{\mathcal Q}(\mathbf{a} \mathbf{1}_{E_k(Q)} w)(x)^{p'} \sigma  \\
&\lesssim \mathcal L_{\ast} w(Q);  
\end{align*}
since $\frac{w(E_k(Q))}{w(Q)^p} \le w(Q)^{1-p}$ we obtain 
$
I_{4,1}(k,m,Q) \lesssim \mathcal L_{\ast}^{\frac{p}{p'}} w(Q)^{\frac{p}{p'}-p+1} = \mathcal L_{\ast}^p;
$ as a result we need only consider the sum
\begin{align*}
\sum_{k \in \mathbb{Z}} \sum_{Q \in \mathcal Q^3_k} \sum_{\substack{I \in \mathcal N(k,m,Q) \\ I \in \mathcal P}} \sigma(R) (\langle f \rangle^{\sigma}_R)^p.
\end{align*}
To finish the proof, we need a uniform bound on the number of times a cube $R$ may appear in the above sum. Consider the following lemma, whose proof we momentarily postpone. 
\begin{lemma} \label{count}
Fix a cube $ R$ which satisfies $R \in \mathcal Q_j$ for some integer $j$, and for $ 1\le l  \le D(R) $ suppose 
\begin{itemize}
\item[\rm{(i.)}] there is an integer $ k_l$ and  $ Q _{l } \in \mathcal Q _{k_l } ^{3}$ with $ R\in \mathcal R_{k_l}  (Q)$, 
\item[\rm{(ii.)}] the pairs $ (Q _{l }, k_l )$ are distinct. 
\end{itemize}
We then have that $ D(R)  \lesssim 1  $, with the implied constant depending upon the dimension, and $ 
\eta $, the small constant previously mentioned. 
\end{lemma}
\noindent Using Lemma \ref{count} and the Carleson embedding theorem, we may estimate
\begin{align*}
\sum_{k \in \mathbb{Z}} \sum_{Q \in \mathcal Q^3_k} \sum_{\substack{R \in N(k,m,Q) \\ R \in \mathcal P}} \sigma(R) (\langle f \rangle^{\sigma}_R)^p &\lesssim \int_{\mathbb{R}^n} f(x)^p \sigma 
\end{align*}
to complete the proof modulo Lemma \ref{count}.
\subsubsection{Proof of Lemma \ref{count}} 
\indent Fix $R \in \mathcal D$ such that there exists 
$k_1,\cdots, k_{D(R)} \in \zat$ and cubes $Q_{1},\cdots,Q_{D(R)}$ so that $R \in \mathcal R_{k_j}(Q_{j})$ for all $1 \leq j \leq D(R)$ and the pairs $(Q_j, k_j)$ are distinct. We argue by contradiction that $D(R) \lesssim 1$. The dyadic structure of $\mathcal D$ immediately implies that by possibly reordering we must have the following
\begin{eqnarray}
Q_{1} \subseteq Q_{2} \subseteq \cdots \subseteq Q_{D(R)} \label{e.tower}.
\end{eqnarray}
Then, we have $R \subset Q^{(1)}_{j}$ for each $j$ by $(\ref{e.Whit})$. At this point we consider two cases; namely 
\begin{itemize}
\item[(a.)] $Q_{1} \subsetneq Q_{2} \subsetneq \cdots \subsetneq Q_{D(R)}$ 
\item[(b.)] $Q_{1} = \cdots = Q_{D(R)} $ .
\end{itemize}
\indent First we want to inspect case (a.). We may assume that $k_1 > \cdots > k_{D(R)}$ by $(\ref{e.Whit})$ (Whitney condition); also it is clear that case (a.) implies
\begin{eqnarray*} R \subset Q^{(1)}_{1} \subset \cdots \subset Q^{(1)}_{D(R)} . \end{eqnarray*}
Hence, by the above and the definition of $\mathcal R_{k_1}$ and $\mathcal R_{k_{D(R)}}$, $R \in \mathcal Q _{k_1 + 3}$ and $ R \in \mathcal Q_{k_{D(R)} +3 }$. We conclude $R \in \mathcal Q_{l} $ for $k_{D(R)} + 3 \leq l \leq k_1 +3$. 
Since we are assuming that $D(R) \lesssim 1$ fails, without loss of generality we may take $ D(R) = 7$. Then we have $R, Q_{7} \in Q_{k_7}$ : 
\begin{align*}
R &\subset Q^{(1)}_{1} \subsetneq \cdots \subsetneq Q^{(1)}_{7} \ \implies \\
R^{(2)} &\subset Q^{(1)}_{7}
\end{align*} 
and this contradicts $(\ref{e.Whit})$. Hence, there is a uniform bound on the number of strict inequalities in $(\ref{e.tower})$, and so we only need to consider (b.). \\
\indent If (b.) holds then by definition we have $w(E_{k_j}(Q_{1})) > \eta w(Q_{1})$ for all $1 \leq j \leq D(R)$. We can without loss of generality assume the $k_i$ are distinct. Then the $E_{k_j}(Q_1)$ are also distinct and 
\begin{eqnarray*}
w(Q_{1}) = \displaystyle\sum_{j \in \zat} w(E_{j}(Q_1)) \geq \displaystyle\sum_{j = 1}^{D(R)} w( E_{k_j}(Q_1)) > \displaystyle\sum_{j=1}^{D(R)} w(Q_1) \eta 
\end{eqnarray*} 
so that it must be $D(R) \leq \eta^{-1}$ and we are done.

The final portion extends our result to spaces of homogeneous type. 
\subsubsection{Acknowledgment} The author would like to thank Dr. Michael Lacey for introducing the problem as well as for his crucial discussions and numerous suggestions. Further, the author would also like to thank Dr. Brett Wick for his indispensable discussions concerning this paper, suggestions, and time. 

\end{document}